\documentclass[reqno]{amsart}%{elsarticle} %
\usepackage{graphicx}
\usepackage{amsmath}
\usepackage[foot]{amsaddr}
\usepackage{bm}
\usepackage{amssymb}
\usepackage{amsfonts}
\usepackage{amsthm}
\usepackage{placeins}
\usepackage{afterpage}
\newtheorem{Theorem}{Theorem}[section]
\newtheorem{Lemma}{Lemma}[section]

\newtheorem{Definition}{Definition}[section]
\newtheorem{Proposition}{Proposition}[section]
\def\ee{\end{eqnarray*}}
\def\be{\begin{eqnarray*}}
\def\bee{\end{eqnarray}}
\def\bbe{\begin{eqnarray}}
\def\ea{\end{align*}}
\def\ba{\begin{align*}}
\makeatletter
\let\today\relax
\def\ps@pprintTitle{%
    \let\@oddhead\@empty
    \let\@evenhead\@empty
    \def\@oddfoot{\footnotesize\itshape
      \hfill\today}%   {Preprint submitted} \hfill\today}%
    \let\@evenfoot\@oddfoot
    }
\makeatother

\makeatletter
\newcommand{\bal}{\@ifstar{\@bals}{\@bal}}
\def\@bals#1\eal{\begin{align*}#1\end{align*}}
\def\@bal#1\eal{\begin{align}#1\end{align}}
\makeatletter
\def\A{A} %\mathbf{A}}
\def\L{ L}%\mathbf{L}}
\def\E{\hat{E}} %\bm E} %\mathbf{E}}
\def\g{ g}
\def\v{v} %\bm v}
\def\u{u}%\bm u}

\def\a{a} %\bm a}

    \def\p{\partial}

\DeclareMathOperator*{\esssup}{ess\,sup}

\begin{document}

\title[Fractional Boussinesq equation in an exterior domain]{Stability of purely  convective steady-states of  fractional Boussinesq equations in  an exterior domain  }

\author{Zhi-Min Chen}

\address{School  of Mathematical Sciences, Shenzhen University, Shenzhen 518060,  China}%
\email{zmchen@szu.edu.cn}

\begin{abstract} A thermal convection flow in the three-dimensional unbounded fluid domain exterior to a sphere is considered. The viscosity force is determined by a fractional power of the  Stokes operator. A purely conductive steady state  arises due to the fluid heated from the sphere. A weak solution of the fluid motion problem is obtained and global stability of the steady-state solution  in $L^2$ is provided.

\


\end{abstract}


\maketitle

\section{Introduction}

The traditional Rayleigh-B\'enard thermal convection problem originated from B\'enard \cite{B2} and Rayleigh \cite{B3} is  considered as a type of natural fluid motion in a planar horizontal layer fluid heated from its lower horizontal boundary, and the flow is driven by a buoyant force due to the temperature difference of the horizontal fluid boundaries. In a study of nonlinear behaviours of atmospheric fluid motion, this horizontal layer  of fluid motion  model  was adopted  by Lorenz \cite{L} (see also references therein)  to approximate locally  the atmospheric flow surrounding the spherical earth surface. The present study is concerned with a  thermal convection flow   in the three-dimensional exterior  domain
\be\Omega = \{ x \in R^3; \,\, |x|>1\}
\ee
 heated from the unit sphere $\p\Omega = \{ x\in R^3;\,\, |x|=1\}$.  The incompressible fluid motion is governed by the following system  in a  dimensionless formulation
\bbe \left.\begin{array}{ll}
\p_t \u + \u\cdot \nabla \u = \Delta \u -\nabla p -\alpha \g T,\,\, \nabla \cdot \u=0,\,\, &\label{a1}
\\ \p_t T +\u\cdot \nabla T =\Delta T,\,\,&
\\
 \u|_{\p\Omega}=0,\,\, T|_{\p\Omega}=\alpha. &
\end{array}\right\}
\bee
Here  the time-dependent functions $\u$, $T$ and $p$ represent, respectively,  the  unknown velocity, temperature and pressure. The temperature $T$ varies from $\alpha>0$ on the  bottom boundary $\p \Omega$ to  zero  at the infinity ( $|x| \to \infty$). This gives rise to the buoyant force, which  is determined by  the acceleration due to the gravity expressed as
\bbe \g=g_0\nabla \frac1{|x|},
\bee
defining a large-scale fluid motion. Here, $g_0$ is a gravitational constant. For simplicity,  $g_0=1$ is assumed and $\alpha$ being a constant is adopted. %therwise, $g_0$ can be absorbed into $\alpha$.

Equation  (\ref{a1}) admits a  purely convective steady-state solution   $(\u_0,p_0,T_0)$ expressed as
\bbe \u_0=0,\,\,\,  \nabla p_0= -\alpha \g T_0,\,\, \Delta T_0=0, \,\,T_0|_{\p\Omega}=\alpha,
\bee
and hence
\bbe \u_0=0,\,\,\, T_0=\frac\alpha {|x|}, \,\, p_0(x) = -\frac{\alpha^2}{ 2 |x|^2}.
\bee
 Local stability of the purely steady-state flow was studied by Hishida and Yamada \cite{HY} and  Hishida \cite{H94,H97} with respect to strong solutions, while $L^2$ global stability of the purely steady-state flow was  studied by Chen-Kagei-Miyakawa \cite{Chen1992} with respect to weak solutions.

Let $L^2_\sigma(\Omega)^3$ be the $L^2$-closure of  $C^\infty_{0,\sigma}(\Omega)^3$, the set  of all smooth
solenoidal vector fields with compact support in  $\Omega$.  Let  $P$ denote the bounded projection operator mapping $L^2(\Omega)^3$ onto  $L^2_\sigma(\Omega)^3$ due to the Helmholtz decomposition \cite{Miya82}.
By adopting the perturbation
\bbe (\u, T)= (\u_0,T_0)+ ( \u, \theta)\label{T0}
\bee
 and applying the projection operator $P$ to the momentum equation, equation  (\ref{a1}) can be  written as
\bbe \left.\begin{array}{ll}
\p_t \u +A \u+P\alpha \g\theta=-P(\u\cdot \nabla \u),  &
\\ \p_t \theta -\Delta \theta+\u\cdot \nabla T_0=-\u\cdot \nabla \theta,&
%\\
 %\u|_{\p\Omega}=0,\,\, \theta|_{\p\Omega} =0, &
\end{array}\right\}
\bee
for $\A= -P\Delta$ and the solutions
\be u(t)\in  D(\A)&=& \left\{ \u \in  W^{2,2}(\Omega)^3;  \,\, \u|_{\p\Omega}=0, \,\, \nabla \cdot u=0\right\},\,\,
\\
\theta(t)\in  D(-\Delta )&=& \left\{ \u \in  W^{2,2}(\Omega)^3;  \,\, \u|_{\p\Omega}=0\right\}.
\ee
 By   the spectral representation (see, for example,  Yosida \cite[page 313, Theorem 1]{Yosida}), we have
 fractional Stokes operator
\bbe \A^\gamma  =\int^\infty_0 \lambda^\gamma  d E_\lambda,\,\, 0<\gamma \le 1\label{Asp}\bee
associated with the domain
\be D(\A^\gamma )= \left\{ \u \in L^2_\sigma(\Omega)^3; \,\,\|\u\|_2+ \|A^\gamma \u \|_2 <\infty, \,\, \u|_{\p\Omega}=0\right\},
\ee
where  $E_\lambda$ denotes  the spectral resolution of the unit determined by the  operator $\A$ and  $\|\cdot \|_r$ represents the Lebesque space norm  $\|\cdot \|_{L^r(\Omega)}$.

It is the purpose of the present paper to consider the global  stability of the basic steady-state flow $(\u_0,T_0)$ in the sense $(\u, T)\to (\u_0,T_0)$ or $(u,\theta_0) \to (0,0)$ as $t\to\infty$ in the $L^2$ space when the viscosity force  $-\Delta \u$ is reduced to  that determined by  a fractional Laplacian $A^\kappa$ in the sense of (\ref{Asp}).
More precisely, the fluid motion problem involving the fractional Laplacian   in the exterior domain $\Omega$  is expressed as
\bbe \label{eq1}\label{a3}\left.\begin{array}{ll}
\p_t \u +\A^\kappa \u+P\alpha \g\theta=-P(\u\cdot \nabla \u),  & %\mbox{ in } \Omega\label{a3}\hspace{1.5mm}
\\ \p_t \theta -\Delta \theta+\u\cdot \nabla T_0=-\u\cdot \nabla \theta,\,\,& %\mbox{ in } \Omega\hspace{1.5mm}
\\
 \u|_{\p\Omega}=0,\,\, \theta|_{\p\Omega} =0, & %\mbox{ on } \p\Omega
 \\
 \u|_{t=0}=\a,\,\, \theta|_{t=0}=b.  & %\mbox{ when } t=0
\end{array}\right\}
\bee
Therefore,  it follows from  (\ref{T0}) that  the $L^2$ stability of $(u_0,T_0)$  reduces to the $L^2$-decay problem with respect to any weak solutions  in the following sense:
\begin{Definition}
 Let $(\a,b) \in L^2_\sigma(\Omega)^3\times L^2(\Omega)$ and  $0<\kappa  \le 1$.   A vector field  $(\u,\theta)$ defining  on the domain   $[0,\infty)\times \Omega$  is said to be  a global weak solution of (\ref{eq1}), if for any $\tau>0$,
\be (\u,\theta) \in L^\infty\left(0,\tau; L^2_\sigma(\Omega)^3\times L^2(\Omega)\right)\cap L^2\left(0,\tau; D(A^{\frac\kappa 2})\times W_0^{1,2}(\Omega)\right),\ee

\bbe\lefteqn{\int^\tau_0 (-\langle  \u(t),\p_tv(t)\rangle + \langle A^{\frac\kappa 2}\u(t),A^{\frac\kappa 2}v(t)\rangle + \langle \alpha \g\theta,v(t)\rangle )dt\label{w1}}\hspace{30mm}
\\&=& \int^\tau_0 \langle \u(t)\otimes \u(t),\nabla v(t)\rangle dt+\langle \u(0),v(0)\rangle \nonumber
\bee
and
\bbe\lefteqn{ \int^\tau_0(- \langle  \theta(t),\p_t\vartheta (t)\rangle + \langle \nabla\theta(t),\nabla\vartheta (t)\rangle +\langle \u\cdot \nabla T_0,\vartheta (t)\rangle )dt \label{w2}}\hspace{40mm}
\\&=& \int^\tau_0 \langle \u(t)\theta(t),\nabla \vartheta (t)\rangle dt+\langle \theta(0),\vartheta (0)\rangle, \nonumber
\bee
for the test function  $(v,\vartheta ) \in C^1([0,\tau]; C_{0,\sigma}^\infty(\Omega)^3\times  C_0^\infty(\Omega))$ with $v(\tau)=0$ and $\vartheta (\tau))=0$. Here  $\langle\cdot,\cdot\rangle $ represents the $L^2$ inner product, $C_0^\infty(\Omega)$ denotes the set of all smooth
functions  with compact support in  $\Omega$ and $W^{1,2}_0(\Omega)$ denotes the closure of $C_0^\infty(\Omega)$ in the Sobolev space $W^{1,2}(\Omega)$.
\end{Definition}

We are in the position to state the main result.
\begin{Theorem}\label{Th1} Let $\frac34 <\kappa  \le 1$, $0<\beta<\frac12$ and   $0<\alpha<2^{-\kappa}$.  Then for  $(a,b) \in L^2_\sigma(\Omega)^3\times L^2(\Omega)$, equation  (\ref{a3}) admits  a global weak solution so that
\bbe\|(\u(t),\theta(t))\|_2 \to 0  \,\,\mbox{ as }\,\, t\to \infty.\label{ass1}\bee

Moreover, if  $(v, \vartheta)$ solves  the linearized problem of (\ref{eq1}):
\bbe \left.\begin{array}{ll}
\p_t v+\A^\kappa v+P\alpha \g\vartheta=0, & \label{linear}
\\ \p_t \vartheta -\Delta \vartheta+v\cdot \nabla T_0=0,\,\,&
\\
 v|_{\p\Omega}=0,\,\, \vartheta|_{\p\Omega} =0, &
 \\
 v|_{t=0}=\a,\,\, \vartheta|_{t=0}=b, &
\end{array}\right\}\label{Leq}
\bee
and
\bbe  \|(v(t),\vartheta(t)) \|_2\le Ct^{-\beta},\,\,\, t>0,\label{con1}\bee
then the algebraic decay
\bbe\|(\u(t),\theta(t))\|_2 \le Ct^{-\beta},\,\,\, t>0 \label{ass2}\bee
holds true.
\end{Theorem}

To derive the stability, we assume that $\alpha >0$ is sufficiently small, so that the buoyant force and external influence from nonhomogeneous boundary condition can be controlled by the viscous force and heat diffusion.
If the temperature on the conducting surface is not small, it is infeasible to keep the stability. For example, instabilities  arise in  the Rayleigh-B\'enard thermal convection flow (see Lorenz \cite{L}, Chen and Price \cite{Chen2006} and Rabinowitz \cite{R}).
 If  the thermal convection effect is neglected, i.e., $\alpha =0$ on the sphere $\p\Omega$, equation (\ref{eq1}) with $\kappa=1$ reduces to the exterior Navier-Stokes equations. The corresponding decay estimates displayed in Theorem \ref{Th1} for the exterior Navier-Stokes equations have been given by Borchers and Miyakawa \cite{Miya1992}.

 The motivation on the stability analysis for  exterior domain flows is    from the understanding of fluid motion passing an object or an object moving in a fluid at a stationary speed (see Oseen \cite{O}). This forced problem has an basic steady-state flow (see Finn \cite{F}). The stability of small stationary Navier-Stokes flows has been well studied (see  Borchers and Miyakawa \cite{Miya1992, Miya1995}, Chen \cite{Chen1993}, Hishida \cite{H}, Kozono and Ogawa \cite{K}, Kozono \cite{K1} and Masuda \cite{1975Masuda} ). Decay problems of exterior Navier-Stokes flows were well investigated by  Borchers and Miyakawa \cite{Miya1990}, Han \cite{Han}, He and Miyakawa \cite{He},  Miyakawa \cite{Miya82}, Miyakawa and Sohr \cite{Miya1988}. We also refer to Chen \cite{Chen2012, Chen2019} for potential flows around objects moving in water waves. The algebraic decay estimate (\ref{ass2}) is mainly determined by the energy estimate of the flow. For sharp decay estimates of a fluid flow  in $L^2$ and Hardy spaces, one may refer to Chen and Miyakawa \cite{ChenM}.  However, there is little literature on fractional exterior Boussinesq flows. The present study provides  the first attempt in the  understanding of  the fractional exterior flows.

  The proof of Theorem \ref{Th1} is divided into four sections. Section 2 is a collection of preliminaries; the existence of the weak solution  is given in Section 3;  decay property (\ref{ass1}) is shown in Section 4; algebraic decay property (\ref{ass2}) is derived in Section 5.

\section{Preliminaries}
Let $C$ and $C_n$ be generic constants, which are  independent of the quantities $\u$, $\theta$, $v$, $\vartheta$, $U$, $V$, $t>0$, $x\in \Omega$, but they may depend on initial vector field $(\a,b)$.  However, $C_n$  depends on the integer $n\ge 1$, while $C$ is independent of $n\ge 1$.

Following lemmas are crucial in our analysis.
\begin{Lemma} \label{L01}(Ladyzhenskaya \cite[page 41]{Lady}) There holds the following inequality
\bbe \int_\Omega \frac{|u(x)|^2}{|x|^2} dx \le 4\int _\Omega |\nabla u(x)|^2 dx,\,\,\, u\in C^\infty_0(\Omega).\bee
\end{Lemma}
%With the use of    integration by parts and H\"older inequality, this inequality can derived as
%\be
%\int_\Omega \frac{|u(x)|^2}{|x|^2} dx                                        %&=& \int^{2\pi}_0 %\int^\infty_0 \frac{|u|^p}{r^p}  r^{n-1} dr d\theta
%&=& \int_{|x|=1} \int^\infty_1 u^2  d\rho  d\phi
%\\
%&=& -2\int_{|x|=1}\int^\infty_1(u \p_\rho u )  \rho d\rho  d\phi
%\\
%&\le & 2\left( \int_\Omega \frac{|u(x)|^{2}}{|x|^{2}}dx\right)^{\frac12}\left(\int_\Omega| \nabla u(x)|^2 %dx\right)^{\frac12}.
%\ee
%Here we have  extended the function $u$ by zero outside $\Omega$ and employed spherical coordinate $(\rho, \phi)$.

Let  $H^\frac\gamma2$ be  the completion of $D(A^\frac\gamma2)$ with respect to  the norm $\|A^\frac\gamma2 u\|_2$  for $0\le \gamma \le 1$.
%We shall use the following Sobolev imbedding inequality.
\begin{Lemma}\label{L02}(Miyakawa \cite[Theorem 2.4]{Miya82})
Let  $0\le \gamma \le 1$ and $\frac1p = \frac12 - \frac\gamma 3$. Then $H^\frac\gamma2$ is equal to the complex interpolation space $[L_\sigma^2(\Omega)^3, H^\frac12]_\gamma$ in the sense of \cite{Tr} and
\bbe \|\u \|_p \le 2^\gamma \| \A^{\frac\gamma2} \u\|_2, \,\,  u\in H^\frac\gamma2.\label{miya}
\bee
\end{Lemma}

\begin{Lemma} \label{L03} If  $0<\gamma \le 1$, then we have
\bbe \label{de}\left(\int_\Omega  \left|\frac {\u(x)}{x}\right|^2dx\right)^{\frac12}\le 2 ^\gamma \|\A^{\frac\gamma2}\u\|_2,\,\,\, \u \in H^{\frac\gamma2}.
\bee

\end{Lemma}
\begin{proof}
%For $0\le \gamma \le 1$, Following Miyakawa  \cite[Definition 2.2]{Miya82}, we denote by $H^\gamma$ the %completion of $D(A^\gamma)$  with respect to the norm $\|A^\gamma \u \|_2$ so that $H^\gamma$ is larger %than $D(A^\gamma)$, since $\Omega$ is an exterior domain and $D(A^\gamma)$ is in its graph norm.
Define the linear operator $\mathcal{T}$ so that
\bbe  \mathcal{T}u= \frac{u(x)}{|x|}.\bee
Since $|x|>1$ in $\Omega$, we see that
\bbe\label{s1} \mathcal{T}: L_\sigma^2(\Omega)^3 \mapsto L^2(\Omega)^3 \,\, \mbox{ and }\,\,  \|\mathcal{T}u\|_2 \le \|u\|_2 .\bee
It follows from Lemma \ref{L01} that
\bbe \label{s2}\mathcal{T}: H^\frac12\mapsto L^2(\Omega)^3 \,\, \mbox{ and }\,\, \|\mathcal{T}u\|_2 \le 2 \|A^\frac12 u\|_2.\bee
By  Lemma \ref{L02}, we have  the complex interpolation
\bbe \label{s3} [L_\sigma^2(\Omega)^3, \, H^\frac12]_\gamma = H^\frac\gamma2.\bee
Hence,  combining  of (\ref{s1})-(\ref{s3}) and the complex interpolation property that the corresponding interpolation functor is exact and of type $\gamma$ in the following sense  (see Triebel \cite[Theorem 1.9.3]{Tr})
\bbe \| \mathcal{T} u \|_2 \le \|\mathcal{T}\|_{L^2_\sigma \mapsto L^2}^{1-\gamma} \|\mathcal{T}\|_{H^\frac12\mapsto L^2}^\gamma \|A^\frac\gamma2 u\|_2\le 2^\gamma \|A^\frac\gamma2 u\|_2,
\bee
we thus have the desired estimate.
\end{proof}

%\section{A linearized   operator}
Define the linear  operator
\bbe \nonumber\L U &=&\left( \A^\kappa \u+\alpha P \g\theta,  -\Delta\theta+\u\cdot \nabla T_0\right)
\\
&=&\left(  \A^\kappa  \u+\alpha P (\theta\nabla\frac1{|x|}), -\Delta\theta+\alpha\u\cdot \nabla \frac1{|x|} \right)
\bee
 for $0<\kappa\le 1$ and $U=(\u,\theta)$ in  the domain
 $D(\L) =D(A^\kappa)\times D(-\Delta)$.

To show the self-adjointness of the operator $L$, we employ the Kato--Rellich Theorem.
\begin{Lemma} \label{LL2}( \cite[Theorem 4.3, page 287]{Katobook}) Let $X$ be a Hilbert space, $A_1: X\mapsto X$ a self-adjoint operator and $A_2: X\mapsto X$ a symmetric  operator with the domain relationship $D(A_1) \subset D(A_2)$.
Assume that
\bbe \label{KR}\|A_2u\|_X \le \delta \| A_1u\|_X +C \| u\|_X,\,\, u\in D(A_1)
\bee
for a nonnegative   constant $\delta <1$.
Then $A_1+A_2$ is a self-adjoint operator.
\end{Lemma}

\begin{Lemma} \label{L04} The operator $\L$ is self-adjoint
and   satisfies the positive property
\bbe \langle L  (u,\theta),(u,\theta)\rangle \, \ge (1-\alpha 2^\kappa)(\|A^{\frac\kappa 2}  \u\|_2^2+\|\nabla  \theta\|_2^2),
\bee
provided that $(u,\theta)\in D(\L)$ and  $0<\alpha <2^{-\kappa}$.
\end{Lemma}
\begin{proof} Similar positivity  result with $\kappa=1$ was given in \cite{Chen1992}. However, we would like to provide a straightforward  proof for reader's convenience.

For $(\u,\theta)\in D(\L)$, we use integration by parts and H\"older inequality  to produce the positivity
\bbe \langle \L (\u,\theta),(\u,\theta) \rangle &=&
\langle A^\kappa   \u+ \alpha \theta\nabla\frac1{|x|},\u\rangle +\langle -\Delta \theta+\alpha\u\cdot \nabla \frac1{|x|},\theta\rangle \nonumber
\\
&=&
\|A^{\frac\kappa 2}  \u\|_2^2+\|\nabla  \theta\|_2^2+2\alpha \langle \u\cdot \nabla \frac1{|x|},\theta\rangle \nonumber
\\ &=&\|A^{\frac\kappa 2}  \u\|_2^2+\|\nabla  \theta\|_2^2-2\alpha \langle  \frac\u{|x|},\nabla\theta\rangle \nonumber
%\\ &\ge &\|A^{\frac\kappa 2}  \u\|_2^2+\|\nabla  \theta\|_2^2-2\alpha %\|\frac\u{|x|}\|_2\|\nabla\theta\|_2\label{y1}
\\ &\ge &\|A^{\frac\kappa 2} \u\|_2^2+\|\nabla  \theta\|_2^2-2\alpha 2^\kappa  \|A^{\frac\kappa 2}u\|_2\|\nabla\theta\|_2\nonumber
\\ &\ge &\left(1-\alpha 2^\kappa \right) (\|A^{\frac\kappa 2}  \u\|_2^2+\|\nabla  \theta\|_2^2),\nonumber
\bee
where we have used Lemma \ref{L03}.

On the other hand, in order to apply  the Kato-Rellich Theorem for the self-adjointness, we adopt the operator decomposition $L=L_1+L_2$  so that
\bbe L_1(\u,\theta) = (A^\kappa u, -\Delta\theta ) \mbox{ and } L_2(\u,\theta) = ( \alpha P(\theta\nabla \frac1{|x|}), \alpha\u\cdot \nabla \frac1{|x|})
\bee
for $D(L_1) = D(A^\kappa)\times D(-\Delta)$ and $D(L_2)= L_\sigma^2(\Omega)^3 \times L^2(\Omega)$.

It is readily seen that the operator $L_2$ has the symmetric property
\bbe
\langle L_2 U, V\rangle =\langle   \alpha \theta\nabla\frac1{|x|}, v\rangle +\langle  \alpha\u\cdot \nabla \frac1{|x|}, \vartheta \rangle
 =\langle  U, L_2V\rangle
\bee
for $U=(\u,\theta), V=(v,\vartheta) \in D(L_2)$. Furthermore, since $|x|\ge 1$, we have

\be \| L_2 U\|_2 &\le& \alpha \|\frac U {|x|^2}\|_2 \le \alpha \| U\|_2,\,\, U\in D(L_1).
\ee
This shows the validity of (\ref{KR}) and hence the self-adjointness of $L$ due to Lemma \ref{LL2}.

\end{proof}

\begin{Lemma}\label{L05}If $0<\kappa  \le 1$ and $0<\alpha < 2^{- \kappa} $, then the operator $\L$ generates  a strongly continuous    analytic semigroup defined as  $e^{-t \L  }$ so that, for $U\in L^2_\sigma(\Omega)^3\times L^2(\Omega)$ and $n=0,1$,
\bbe\|\L^n   e^{-t\L } U\|_2 \le t^{-n} \|U\|_2,\,\,\,\,\lim_{t\to 0}\|  e^{-t\L } U-U\|_2 =0,\label{X1}
\bee
\bbe\label{XXX1} \|A^{\frac\kappa 2} P_1 e^{-t\L } U\|_2^2+\|\nabla P_2 e^{-t\L } U\|_2^2\le \frac{t^{-1}}{1-\alpha  2^\kappa  } \|U\|_2^2,
\bee
\bbe \lim_{t\to \infty}\|e^{-t\L } U\|_2=0,\label{X2}
\bee
Here  the projection operators are defined as
\bbe P_1(u,\theta)=(\u,0) \,\, \mbox{ and }\,\, P_2(u,\theta)=(0,\theta)\bee
for $(\u,\theta)\in L^2_\sigma(\Omega)^3\times L^2(\Omega)$.
\end{Lemma}

\begin{proof}
 By Lemma \ref{L04}, we have the   spectral representation (see, for example,  Yosida \cite[Page 313, Theorem 1]{Yosida})
\bbe \L = \int^\infty_0 \lambda d\E_\lambda
\bee
with $\E_\lambda$ the spectral resolution of the unit determined by the operator $L$.
Consequently, equation  (\ref{X1}) is verified by
\bbe \|\L  ^ne^{-t\L } U\|_2^2= \int^\infty_0 \lambda^{2n} e^{-2t\lambda} d\|\E_\lambda U\|_2^2\le t^{-2n} \|U\|_2^2,\,\,\, n=0, 1, \label{semi}
\bee
and
\begin{eqnarray*}
\|(e^{-t \L}-1) U\|_2 &\le& \|(e^{-t \L}-1) (U-V)\|_2+\|(e^{-t \L}-1) V\|_2
\\
&\le& \|U-V\|_2+t\left( \int^\infty_0 \lambda^2 d \|\E_\lambda V\|^2_2\right)^\frac12= \|U-V\|_2+ t \|L V\|_2
\end{eqnarray*}
for $V\in C^\infty_{0,\sigma}(\Omega)^3\times C_0^\infty(\Omega)$.
Equation (\ref{X1}) ensures the strongly continuous  analytic semigroup  property.

For the derivation of (\ref{XXX1}), by Lemma \ref{L04} and H\"older inequality,
we have
\bbe (1-\alpha  2^\kappa  ) \left(\| A^{\frac\kappa 2}P_1e^{-t\L }U\|_2^2+\|\nabla P_2e^{-t\L }U\|_2^2\right) &\le & \langle  \L   e^{-t\L }U,e^{-tL}U\rangle\nonumber
\\
&\le& \|\L   e^{-t\L }U\|_2\|e^{-t\L }U\|_2\nonumber
\\
&\le&t^{-1} \|U\|_2^2.\label{kappa}
\bee

For the validation of the decay property (\ref{X2}), we take $V=(v,\vartheta)\in C^\infty_{0,\sigma}(\Omega)^3\times C^\infty_0(\Omega)$, $V' \in L^2_\sigma(\Omega)^3\times L^2(\Omega)$ and $\frac1r=\frac12-\frac\kappa3$. By H\"older inequality, (\ref{miya}), Sobolev imbedding and (\ref{XXX1}),  we see that
 \be \langle e^{-t\L}V, V'\rangle &\le& \|v\|_r \|P_1e^{-t\L}V'\|_q+ \|\vartheta\|_{\frac65}\|P_2e^{-t\L} V'\|_6
 \\
 &\le& C\|v\|_r \|A^\frac\kappa2P_1 e^{-t\L}V'\|_2+ C\|\vartheta\|_{\frac65}\|\nabla P_2e^{-t\L} V'\|_2
 \\
 &\le& Ct^{-\frac12}(\|v\|_r +\|\vartheta\|_{\frac65})\| V'\|_2.
 \ee
 This together with (\ref{X1})  yields
 \bbe \|e^{-t \L}U\|_2 &\le& \|e^{-t\L} (U-V)\|_2 +\|e^{-t\L}V\|_2\nonumber
 \\
 &\le& \|U-V\|_2 +Ct^{-\frac12}(\|v\|_r+ \|\vartheta\|_\frac65).
 \bee
 This implies (\ref{X2}), after taking the limit $t\to \infty$ and using the density of
 the set $C^\infty_{0,\sigma}(\Omega)^3\times C^\infty_0(\Omega)$ in $ L^2_\sigma(\Omega)^3\times L^2(\Omega)$.

\end{proof}

\begin{Lemma} \label{L06} Let $0<\kappa \le 1$,  $0<\alpha\le  2^{-\kappa} $ and $U=(\u,\theta)\in L^2_\sigma(\Omega)^3\times L^2(\Omega)$. Then we have
\bbe\|A^\kappa  P_1 e^{-t\L }U \|_2\le  \left(t^{-1} +t^{-\frac12}\frac{2\alpha}{\sqrt{1-\alpha  2^\kappa  }}     \right)   \|U \|_2.\label{o1}\bee
\end{Lemma}
\begin{proof}
Since $|x|>1$ in $\Omega$, it follows from Lemmas \ref{L01} and \ref{L04} that
\bbe \|A^{\kappa } \u\|_2&\le & \|A^{\kappa } \u + \alpha P(\theta \nabla \frac1{|x|})\|_2+\alpha \|  P(\theta \nabla \frac1{|x|})\|_2\nonumber
\\
&\le & \|A^{\kappa } \u + \alpha P(\theta \nabla \frac1{|x|})\|_2+\alpha \|  \frac\theta{|x|^2}\|_2\nonumber
\\
&\le & \|A^{\kappa } \u + \alpha P(\theta \nabla \frac1{|x|})\|_2+\alpha\|  \frac\theta{x}\|_2\nonumber %\label{new3}
\\
&\le & \|\L U\|_2+ 2\alpha \|\nabla \theta\|_2\nonumber
\\
&\le& \|\L U\|_2 + \frac{2\alpha}{\sqrt{1-\alpha  2^\kappa  }}\sqrt{\langle \L U,U\rangle}. \label{new1}
\bee
 Thus (\ref{o1}) is obtained from the combination of (\ref{X1}) and the following inequality
\bbe\label{fin1}
\|A^{\kappa } P_1e^{-t\L} U\|_2 &\le& \|\L e^{-t\L}U\|_2 + \frac{2\alpha}{\sqrt{1-\alpha  2^\kappa  }}\sqrt{\| \L e^{-t\L}U\|_2\| e^{-t\L}U\|_2 }.
%\\
%&\le & \left(t^{-1} +t^{-\frac12}\frac{2\alpha}{\sqrt{1-\alpha  2^\kappa  }}\right)\|U\|_2\nonumber
\bee
\end{proof}

\begin{Lemma}\label{L07} For $\frac12\le\kappa  \le 1$,  $0<\alpha\le  2^{-\kappa} $ and $U= (\u,\theta)\in L_\sigma^2(\Omega)^3\times L^2(\Omega)$, we have
\bbe\|\nabla e^{-t\L } U \|_2 &\le & C (t^{-\frac12}+t^{-\frac1{2\kappa }})\|U\|_2.
\bee
\end{Lemma}
\begin{proof}
By Lemmas \ref{L04} and \ref{L05}, we have
\bbe \|\nabla P_2 e^{-t\L } U \|_2&\le &\frac1{\sqrt{1-\alpha 2^\kappa}}\sqrt{\|\L e^{-t\L } U \|_2\| e^{-t\L } U \|_2}\nonumber
\\
&\le& \frac{t^{-\frac12}}{\sqrt{1-\alpha 2^\kappa}}\|U\|_2.\label{D}
\bee
This gives the estimate with respect to the second component. For the estimate with respect to another component, we  employ the  condition  $\frac12 \le \kappa  \le 1$, the spectral representation  and  H\"older inequality to produce
\bbe \|\nabla u\|_2=\|A^{\frac12} \u\|_2
%&=&\left(\int^\infty_0 \lambda d \langle E_\lambda \u,\u\rangle\right)^\frac12\nonumber
%\\
%&\le&  \left(\int^\infty_0 \lambda^{\kappa } d \langle E_\lambda \u,\u\rangle\right)^{2-\frac1\kappa } %\left(\int^\infty_0 \lambda^{2\kappa } d \langle E_\lambda \u,\u\rangle\right)^{\frac1\kappa -1}\nonumber
%\\
&\le& \|A^{\frac\kappa 2}\u\|_2^{2-\frac1\kappa }\|A^\kappa  \u\|_2^{\frac1\kappa -1}.\label{new2}
\bee
Hence, by Lemma \ref{L04} and H\"older inequality, it follows that
\be
\|\nabla P_1 e^{-t\L } U \|_2
&\le &\|A^{\frac\kappa 2}P_1e^{-t\L } U\|_2^{2-\frac1\kappa }\|A^\kappa  P_1e^{-t\L } U\|_2^{\frac1\kappa -1}
\\
&\le &\left(\frac{\langle  \L e^{-t\L}U, e^{-t \L}U\rangle}{1-\alpha 2^\kappa}\right) ^{\frac12(2-\frac1{\kappa })}\|A^\kappa  P_1e^{-t\L } U\|_2^{\frac1\kappa -1}.
\ee
With the use of Lemmas \ref{L05} and \ref{L06}, the previous equation is bounded by
\be
\lefteqn{\left(\frac{t^{-1}}{1-\alpha 2^\kappa}\|U\|_2^2\right)^{\frac12(2-\frac1{\kappa })}\left(\left(t^{-1} +\frac{2\alpha t^{-\frac12}}{\sqrt{1-\alpha  2^\kappa  }}     \right)  \|U\|_2\right)^{\frac1\kappa -1}}%\hspace{40mm}
\\
&\le &Ct^{-1+\frac1{2\kappa }}(t^{-\frac12}+t^{-1})^{\frac1\kappa -1}\|U\|_2
\\
&\le &%Ct^{-\frac12}(1+t^{-\frac12})^{\frac1\kappa -1}\|U\|_2
 C(t^{-\frac12}+t^{-\frac1{2\kappa }})\|U\|_2.
\ee
The proof is complete.
\end{proof}

\section{Existence of weak solutions}

 Rewrite (\ref{a3}) as
\bbe \p_t U + \L U = -(P(\u\cdot \nabla u), u\cdot\nabla \theta),\,\,\,\, U(0) = U_0=(\a,b)\label{xx1}\bee
for $U=(\u,\theta)$.
With the use of the mollification operators
\bbe J_n= n(n+A)^{-1}\,\,\,\mbox{ and }\,\,\, I_n= n(n+\L)^{-1}
\bee
for  positive integers $n$,  we  construct approximation solutions to (\ref{xx1}) by solving the equation
\bbe \p_t U_n + \L U_n = -(P(J_n u_n\cdot \nabla u_n), J_n u_n\cdot\nabla \theta_n);   \,\,\,\, U_n(0) =I_n U_0\label{xx2}\bee
with $U_n = (\u_n,\theta_n)$.
By the spectral representations of $A$ and $L$, we have
\bbe \|J_n \u\|_2\le \|\u\|_2,\,\,\, \|I_n U\|_2\le \|U\|_2, \label{JkIk}\bee
\bbe \lim_{n\to \infty } (\|J_n \u- \u \|_2+ \| I_n U-U\|_2)=0,\bee
and
\bbe \| J_n\u\|_\infty\le C_n \|\u\|_2; \,\,\,\, \| L I_n U\|_2\le C_n \|U\|_2,\label{JkIk1}\bee
after the use of Sobolev imbedding.

\begin{Proposition}
 For  $\frac12 <\kappa \le 1$, $0<\alpha <2^{-\kappa}$,  integer $n\ge 1$, $U_0\in L^2_\sigma(\Omega)^3\times L^2(\Omega)$ and $\tau>0$, then  (\ref{xx2}) admits  a unique solution $U_n=(u_n,\theta_n)$ so that
 \bbe U_n\in  L^2(0,\tau; D(\L))\cap W^{1,2}(0,\tau; L^2_\sigma(\Omega)^3\times L^2(\Omega)).
 \bee
 %$U_n \in C([0,\tau]; W_{0,\sigma}^{1,2}(\Omega)^3\times W_0^{1,2}(\Omega))$.

%(ii) The solution $U_n$ is in $L^2(0,\tau; D(\L))\cap W^{1,2}(0,T; L^2_\sigma(\Omega)^3\times L^2(\Omega))$
\end{Proposition}

\begin{proof}
Let us begin with the  existence of the approximate solutions to the integral equation of (\ref{xx1}):
\bbe\label{int} U_n(t) = e^{-t\L} I_n U_0-\int^t_0 e^{-(t-s)\L} (P(J_n\u_n\cdot \nabla u_n), J_n\u_n\cdot \nabla \theta_n) ds
\bee
in the  space
$ L^\infty(0,\tau; W_{0,\sigma}^{1,2}(\Omega)^3\times W_0^{1,2}(\Omega)),$
 where $ W_{0,\sigma}^{1,2}(\Omega)^3$ denotes the closure of $C^\infty_{0,\sigma}(\Omega)^3$ in the Sobolev space $W^{1,2}(\Omega)^3$.
The solution existence is to be obtained by  the contraction mapping principle in the complete metric space
\be X &=& \left\{ U =(u,\theta)\in L^\infty\left(0,\tau; W_{0,\sigma}^{1,2}(\Omega)^3\times W_0^{1,2}(\Omega)\right); \right. %\,\, U(0) = I_n U_0,\,\,\right.
\\
&&\left.\|U\|_X = \esssup_{0< t< \tau} (\|U(t)\|_2 + \|\nabla  U(t)\|_2)\le M\right\}\ee
 with respect to the operator
\be F_n U(t) = e^{-t\L} I_n U_0-\int^t_0 e^{-(t-s)\L} (P( J_n\u\cdot \nabla u),J_nu\cdot \nabla \theta) ds,\,\, U=(u,\theta).
\ee

By  (\ref{new2}), Lemma \ref{L04} and  the observation $|x|>1$ in $\Omega$,  we have
\begin{align*}\|&\nabla e^{-t \L} I_n U_0\|_2
\\
 &\le \|A^{\kappa /2} P_1e^{-t \L} I_n U_0\|_2^{2-\frac1\kappa } \|A^\kappa  P_1e^{-t \L} I_n U_0\|_2^{\frac1\kappa -1} +\|\nabla P_2e^{-t \L} I_n U_0\|_2
\\
&\le \left(\frac{\|\L e^{-t \L} \! I_n U_0\|_2\| e^{-t \L}\! I_n U_0\|_2}{1-\alpha 2^\kappa}\right)^{1-\frac1{2\kappa }}\!\!\!\!\! \left(\|\L e^{-t \L} \! I_n U_0\|_2 \!+\! \alpha \|P_2 e^{-t \L} \! I_n U_0\|_2\right)^{\frac1\kappa -1}
\\&+\frac1{\sqrt{1-\alpha 2^\kappa}}\|\L e^{-t \L} I_n U_0\|_2^\frac12 \| e^{-t \L} I_n U_0\|_2^\frac12.
\end{align*}
It follows from Lemma \ref{L05}  that  the previous equation is bounded by
\be
C(\|\L I_n U_0\|_2\|  I_n U_0\|_2)^{1-\frac1{2\kappa }} (\|\L  I_n U_0\|_2+  \| I_n U_0\|_2)^{\frac1\kappa -1} +C\|\L  I_n U_0\|_2^\frac12 \|  I_n U_0\|_2^\frac12.\ee
Hence, employing (\ref{JkIk}) and  (\ref{JkIk1}), we obtain
\be
\|\nabla e^{-t \L} I_n U_0\|_2&\le & C_n\|  U_0\|_2,
\ee
and, by  Lemma \ref{L05} and for  a sufficiently large $M$ depending on $n$,
\bbe \| e^{-t \L} I_n U_0\|_2 +\|\nabla e^{-t \L} I_n U_0\|_2 &\le & C_n\|U_0\|_2  \label{new111}
\le \frac M2.
\bee

On the other hand,  for  $U=(\u,\theta)\in  X$, we use Lemma \ref{L05}, Lemma  \ref{L07},   (\ref{JkIk1}) and (\ref{new111})  to produce
\bbe
\lefteqn{\|F_n U(t)\|_2 + \|\nabla F_n U(t)\|_2}\nonumber\\
&\le& \| e^{-t\L} I_n U_0\|_2+\| \nabla e^{-t\L} I_n U_0\|_2\nonumber
+\int^t_0 \|e^{-(t-s)\L}  (P(J_n\u\cdot \nabla u), J_nu\cdot \nabla \theta)\|_2 ds\nonumber
\\
&&+\int^t_0 \|\nabla e^{-(t-s)\L} (P(J_n\u\cdot \nabla u), J_nu\cdot \nabla \theta)\|_2 ds\nonumber
\\
&\le & \frac M2+C\int^t_0 (1+(t-s)^{-\frac12} +(t-s)^{-\frac1{2\kappa }}) \| J_n\u\cdot \nabla U\|_2 ds\nonumber
\\
&\le &\frac M2+C\int^t_0 (1+(t-s)^{-\frac12} +(t-s)^{-\frac1{2\kappa }}) \| J_n\u\|_\infty \| \nabla U\|_2 ds\nonumber
\\
&\le & \frac M2+C_n\int^t_0 (1+(t-s)^{-\frac12} +(t-s)^{-\frac1{2\kappa }}) \|\u\|_2 \|\nabla U\|_2 ds\label{1oo1}
\\
&\le &\frac M2+ C_n(\tau+\tau^{\frac12} +\tau^{1-\frac1{2\kappa }}) M^2\le M,\nonumber
\bee
provided that $\tau$ is sufficiently small.
Similarly, for  $U,\,U'\in X$, we have the contraction property
\be
\|F_n U-F_n U'\|_X
&\le & C_n(\tau +\tau^{1/2}+\tau^{1-\frac1{2\kappa }}) M\|U-U'\|_X\le \frac12 \|U-U'\|_X,
\ee
when $\tau$ is sufficiently small. Thus $F_n$ is a contraction operator mapping $ X$ into itself, and so $F_n$ admits a unique fix point $U_n\in  X$ satisfying  integral equation (\ref{int}). We thus obtain local existence result.

To show the solution  global existence, we  adopt  the $L^2$ theory  of linear parabolic equations \cite{Simon} or Miyakawa and Sohr \cite{Miya1988} to obtain that
\be
U_n \in L^2\left(0, \tau; D(\L )\right) \mbox{ and } \p_t U_n \in
L^2\left(0, \tau; L^2_\sigma(\Omega)^3 \times L^2(\Omega)\right)
\ee
 and $U_n$ solves     the approximate equation (\ref{xx2}). To give the local solution extension, we have to provide the boundedness of $\|U_n(t)\|_X$  for $t\in[0,\tau)$ with respect to  any given constant  $\tau>0$.

 Indeed, we obtain from (\ref{xx2}) and the divergence free condition of $J_n\u$ that
\bbe \langle \p_t U_n, U_n\rangle  + \langle \L U_n, U_n\rangle  = -\langle  (J_n \u_n\cdot \nabla U_n),U_n\rangle =0.  \label{xx2x}
\bee
Hence, after integration with respect to $t$ and the use of Lemma \ref{L04},
\bbe
\|U_n(t)\|_2^2 +2(1-\alpha  2^\kappa  )\int^t_0(\|A^\frac\kappa 2\u_n(s)\|_2^2+\|\nabla \theta_n(s)\|_2^2)ds\le  \|U_n(0)\|_2^2\le \|U_0\|_2^2.\label{new112}
\bee
This shows the uniform boundedness of $\|U_n(t)\|_2$ for $t\ge 0$.

For the boundedness of $\|\nabla U_n(t)\|_2$,  we use (\ref{xx2}),   (\ref{new111}),  (\ref{new112}) and the derivation of (\ref{1oo1})  to produce
\be
\|\nabla U_n(t)\|_2&\le & C_n\|U_0\|_2+C_n\int^t_0 ((t-s)^{-\frac12} +(t-s)^{-\frac1{2\kappa }}) \|\u_n(s)\|_2 \|\nabla U_n(s)\|_2 ds
\\
&\le & C_n+C_n\int^t_0 ((t-s)^{-\frac12} +(t-s)^{-\frac1{2\kappa }})  \|\nabla U_n(s)\|_2 ds.
 \ee
 Hence,  we have the estimate involving an exponential-type function  expressed as 
 \bbe\label{kkka} \|\nabla U_n(t)\|_2 
&\le&  C_n\sum_{m=0}^\infty  \frac{ (C_n (\tau^{1-\frac1{2\kappa}}+\tau^\frac12) )^m}{\Gamma ( m(1-\frac1{2\kappa}) +1)}, \,\,\, 0\le t <\tau
 \bee
 which gives the desired bound after taking  $t \to \tau$. The validity of (\ref{kkka}) is due to Lemma \ref{gg} in the following. 
\end{proof}

\begin{Lemma}\label{gg} Let $\phi \in L^\infty(0,\sigma)$ be subject to the inequality
\bbe\label{k1} \phi(t) \le a_1 +a_2 \int^t_0 ((t-s)^{\delta-1}+(t-s)^{\rho-1})\phi(s) ds, \,\, 0\le t\le \sigma
\bee
for  positive constants $a_1$, $a_2$, $\sigma$, $\rho$ and $\delta$ with     $\delta \ge \rho$. Then, we have
\bbe \label{ka} \phi(t) &\le &a_1\sum_{n=0}^{\infty}  \frac{ (a_2\Gamma(\rho) (\sigma^\delta+\sigma^\rho) )^n}{\Gamma ( n\rho +1)},\,\, 0\le t\le \sigma.
\bee
\end{Lemma}
This Gronwall inequality is implied from  \cite[Lemma 7.1.1, page 188]{Henry}. For reader's convenience, we sketch its derivation.
\begin{proof}
It follows from (\ref{k1}) that 
\bbe \phi(t) \le a_1 +K\phi(t) \,\,\,\mbox{ for } K\phi(t) \doteq a_2 (\sigma^{\delta-\rho}+1)\int^t_0 (t-s)^{\rho-1})\phi(s) ds.
\bee
Hence for any  integer $m>0$, iterating  repeatedly, we have
\begin{align*} \phi(t) \le &\sum_{n=0}^{m} K^n a_1(t) +K^{m+1}\phi(t)
\\
\le &a_1\sum_{n=0}^{m}  \frac{ (a_2\Gamma(\rho) (\sigma^{\delta-\rho}+1) )^n t^{\rho n}}{\Gamma ( n\rho +1)}+a_1  \frac{ (a_2\Gamma(\rho) (\sigma^{\delta-\rho}+1) )^{m+1} t^{\rho (m+1)}}{\Gamma ( (m+1)\rho +1)}\|\phi\|_{L^\infty(0,\sigma)}
\end{align*}
This gives (\ref{ka}) after taking $m\to \infty$. 
\end{proof}

\begin{Proposition}\label{TT1}
For $\frac34<\kappa  \le 1$, $0<\alpha< 2^{-\kappa}$ and  $U_0=(\a,b)\in L^2_\sigma(\Omega)^3\times L^2(\Omega)$, then (\ref{eq1}) admits a global weak solution.
\end{Proposition}
\begin{proof}
To obtain the  weak solution existence, we need to study the compactness of the sequence $\{U_n\}_{n\ge 1}$. For
$ V=(v,\vartheta )\in W^{1,2}_{0,\sigma}(\Omega)^3\times W^{1,2}_0(\Omega),$
 we see that
\be \langle \p_tU_n, V\rangle &=&-\langle \L U_n, V\rangle  -\langle J_n\u_n\cdot \nabla U_n,V\rangle
\\ &=& -\langle  A^{\kappa -\frac12}\u_n, A^\frac12 v\rangle -\langle \nabla \theta_n, \nabla \vartheta \rangle -\langle \alpha \theta_n\nabla \frac1{|x|},\v\rangle
\\
&&+\langle \alpha \frac{\u_n}{|x|},\nabla \vartheta \rangle + \langle J_n\u_n\otimes \u_n,\nabla v\rangle + \langle J_n\u_n\theta_n,\nabla \vartheta \rangle
\\ &\le& \left(\| A^{\kappa -\frac12}\u_n\|_2+\|\nabla \theta_n\|_2+
\|J_n\u_n\|_4\|U_n\|_4\right)\|\nabla V\|_2
\\
&&+ \alpha \|\frac{\theta_n}{x} \|_2 \|\frac{v}{x}\|_2 + \alpha \|\frac{\u_n}{x}\|_2\|\nabla \vartheta \|_2
\\ &\le& \left(\| A^{\kappa -\frac12}\u_n\|_2+4\alpha \|U_n\|_2+\|\nabla \theta_n\|_2+
\|J_n\u_n\|_4\|U_n\|_4\right)\|\nabla V\|_2,
\ee
where we have used integration by parts, H\"older inequality and Lemma \ref{L01}.

Hence, by the derivation of (\ref{new2}), H\"older inequality,  Sobolev imbedding and Lemma \ref{L02}, we have
\be \lefteqn{\|\p_t U_n\|_{W^{-1,2}}}\\
&\le& \| \u_n\|_2^{\frac1\kappa -1}\|A^{\frac\kappa 2}\u_n\|_2^{2-\frac1\kappa }+\|\nabla \theta_n\|_2+4\alpha\|U_n\|_2
\\&&+\|J_n\u_n\|_2^{1-\frac3{4\kappa }}\|J_n\u_n\|_{1/(\frac12-\frac\kappa 3)}^{\frac3{4\kappa }}(\|\u_n\|_2^{1-\frac3{4\kappa }}\|\u_n\|_{1/(\frac12-\frac\kappa 3)}^{\frac3{4\kappa }}+\|\theta_n\|_2^\frac14\|\theta_n\|_6^\frac34)
\\
&\le& \| \u_n\|_2^{\frac1\kappa -1}\|A^{\frac\kappa 2}\u_n\|_2^{2-\frac1\kappa }+\|\nabla \theta_n\|_2+4\alpha\|U_n\|_2
\\&&+
3\|J_n\u_n\|_2^{1-\frac3{4\kappa }}\|A^{\frac\kappa 2}J_n\u_n\|_{2}^{\frac3{4\kappa }}(\|\u_n\|_2^{1-\frac3{4\kappa }}\|A^\frac\kappa 2\u_n\|_{2}^{\frac3{4\kappa }}+\|\theta_n\|_2^\frac14\|\nabla \theta_n\|_2^\frac34).
\ee
This together with (\ref{JkIk}) and (\ref{new112}) implies that
\be \lefteqn{\|\p_t U_n\|_{W^{-1,2}}}\\
&\le&  \| \u_n\|_2^{\frac1\kappa -1}\|A^{\frac\kappa 2}\u_n\|_2^{2-\frac1\kappa }+\|\nabla \theta_n\|_2+4\alpha \|U_n\|_2
\\
&&+
3\|\u_n\|_2^{1-\frac3{4\kappa }}\|A^{\frac\kappa 2}\u_n\|_{2}^{\frac3{4\kappa }}(\|\u_n\|_2^{1-\frac3{4\kappa }}\|A^\frac\kappa 2\u_n\|_{2}^{\frac3{4\kappa }}+\|\theta_n\|_2^\frac14\|\nabla \theta_n\|_2^\frac34)
\\
&\le&  \| U_0\|_2^{\frac1\kappa -1}\|A^{\frac\kappa 2}\u_n\|_2^{2-\frac1\kappa }+\|\nabla \theta_n\|_2+4\alpha \|U_0\|_2
\\
&&+
3\|U_0\|_2^{1-\frac3{4\kappa }}\|A^{\frac\kappa 2}\u_n\|_{2}^{\frac3{4\kappa }}(\|U_0\|_2^{1-\frac3{4\kappa }}\|A^\frac\kappa 2\u_n\|_{2}^{\frac3{4\kappa }}+\|U_0\|_2^\frac14\|\nabla \theta_n\|_2^\frac34)
\\&\le& C (\|A^{\frac\kappa 2}\u_n\|_2^{2-\frac1\kappa }+\|\nabla \theta_n\|_2+1+
\|A^\frac\kappa 2\u_n\|_{2}^{\frac3{2\kappa }}+\|A^{\frac\kappa 2}\u_n\|_{2}^{\frac3{4\kappa }}\|\nabla \theta_n\|_2^\frac34).\ee
Thus consulting with (\ref{new112}), we obtain  that
\bbe \p_t U_n\in L^{\frac{4\kappa }3}(0,\tau; W^{-1,2}(\Omega)^4) \,\,\mbox{ uniformly bounded }\label{abc1}
\bee
with respect to  all $n>0$ and
 for any fixed  $\tau>0$. Moreover, equation  (\ref{new112}) also gives rise to uniform boundedness of  the sequence  $\{U_n\}_{n\ge 1}$   in the space
\bbe L^\infty(0,\tau; L^2_\sigma(\Omega)^3\times L^2(\Omega))\cap L^2 (0,\tau; D(A^{\frac\kappa 2})\times W_0^{1,2}(\Omega)).\label{space}
\bee
Therefore,
 $U_n$ admits a subsequence, denoted again by $U_n$, converging to an element   $U$ in the space given by (\ref{space}) in the following sense:
\bbe \label{abc2}
&&U_n  \rightarrow U \,\,\mbox{weak-star  in }\,\, L^\infty(0,\tau; L^2_\sigma(\Omega)^2\times L^2(\Omega)),
\\ \label{abc3}
&& U_n  \rightarrow U \,\,\mbox{ weakly in }\,\, L^2 (0,\tau; D(A^{\frac\kappa 2})\times W_0^{1,2}(\Omega)).
\bee
 Furthermore, by the compactness result \cite[Theorem 2.1 on page 271 \& Remark 3.2 on page 290]{T}), we see that
\be  U_n  \rightarrow U \,\,\mbox{ strongly in }\,\, L^2 (0,\tau; L^2(K)^4) \mbox{ for any compact set $K \subset \Omega$}.
\ee
due to (\ref{abc1}), (\ref{abc2}) and (\ref{abc3}).
This convergence implies immediately that   $U$ is a desired weak solution  solving (\ref{w1}) and (\ref{w2}).

The proof of Proposition \ref{TT1} is complete.
\end{proof}

\section{Proof of $L^2$-decay property (\ref{ass1})}

From the convergence of $U_n\to U$ in the sense described in  the previous subsection, we see that
\be \|U(t)\|_2 \le \liminf_{n\to \infty}\|U_n(t)\|_2
\ee
for $a.e.$ $t>0$. Thus it  suffices  to  show the corresponding $L^2$ decay estimates of $U_n$ uniformly with respect to $n$.  Therefore, for simplicity of notation, we omit the subscript $n$ for the approximate solutions so that $U=(\u,\theta)$ represents for $U_n=(\u_n,\theta_n)$. Thus  approximate  equations
(\ref{xx2}) and (\ref{int}) are  rewritten, respectively,  as
\bbe
\p_t U + \L U = -(P (J_n\u\cdot \nabla u), J_nu\cdot\nabla \theta), \,\,\, U(0)=I_n U_0,\label{xxx1}
\bee
and
\bbe
\label{intt} U(t) = e^{-t\L} I_n U_0-\int^t_0 e^{-(t-s)\L}(P (J_n\u\cdot \nabla u), J_nu\cdot\nabla \theta) ds.
\bee
According to the proof of (\ref{new112}),   energy inequality (\ref{new112}) can be rephrased as
\bbe \|U(t)\|_2^2 + (1-\alpha  2^\kappa  )\int^t_s(\|A^{\frac\kappa 2}\u(r)\|_2^2+\|\nabla \theta(r)\|_2^2)dr\le \|U(s)\|_2^2\label{energy}\bee
for  $0\le s <t$.

Let us begin with the estimate of the nonlinear integrand of (\ref{intt}) in the $L^2$ norm.
Applying the divergence free property   and  integrating by parts, we take a test function $V\in L^2_\sigma (\Omega)^3\times L^2(\Omega)$ to obtain
\bbe \lefteqn{\langle e^{-t\L }(P (J_n\u\cdot \nabla u), J_nu\cdot\nabla \theta), V\rangle}\nonumber\\
 &=& -\langle J_n\u\otimes \u,\nabla P_1 e^{-t\L } V\rangle  - \langle (J_n\u)\theta,\nabla P_2 e^{-t\L } V\rangle. \label{hh1}
 \bee
 By (\ref{D}), Lemma \ref{L07} and H\"older estimate, the previous equation is bounded by
 \bbe
\lefteqn{ \| J_n\u\|_4 \| \u\|_4 \|\nabla P_1e^{-t\L } V\|_2+ \|J_n\u\|_3\|\theta\|_6 \|\nabla P_2e^{-t\L } V\|_2}\nonumber\\\nonumber
\\
&\le & C(t^{-\frac12}+t^{-\frac1{2\kappa }})\| J_n\u\|_2^{1-\frac3{4\kappa }} \| J_n\u\|_{\frac1{\frac12-\frac\kappa 3}}^{\frac3{4\kappa }}\| \u\|_2^{1-\frac3{4\kappa }} \| \u\|_{\frac1{\frac12-\frac\kappa 3}}^{\frac3{4\kappa }} \|V\|_2\nonumber
\\
&&+C t^{-\frac12}\|J_n\u\|_2^{1-\frac1{2\kappa }}\|J_n\u\|_{\frac1{\frac12-\frac\kappa 3}}^{\frac1{2\kappa }}\|\theta\|_6\| V\|_2. \label{hh2}
\bee
With the use of  (\ref{JkIk}),  Sobolev imbedding and Lemma \ref{L02} for $\frac1p = \frac12-\frac \kappa 3$, the combination of (\ref{hh1})-(\ref{hh2}) becomes
\bbe \|e^{-t\L }(P J_n\u\cdot \nabla u, J_nu\cdot \nabla \theta)\|_2  &\le & C(t^{-\frac12}+t^{-\frac1{2\kappa }})\| \u\|_2^{2-\frac3{2\kappa }} \| A^{\frac\kappa 2}\u\|_2^{\frac3{2\kappa }}\nonumber
\\
&&+ Ct^{-\frac12}\|\u\|_2^{1-\frac1{2\kappa }}\|A^{\frac\kappa 2}\u\|_{2}^{\frac1{2\kappa }}\|\nabla \theta\|_2.\label{oo1}
\bee
Applying the previous estimate into   (\ref{intt}), we have
\bbe
 \|U(t)\|_2&\le &  \| I_n e^{-t\L}U_0\|_{2}\nonumber
+C\int^t_0 (t\!-\!s)^{-\frac12} \|\u(s)\|_2^{1-\frac1{2\kappa }}\|A^{\frac\kappa 2}\u(s)\|_{2}^{\frac1{2\kappa }} \| \nabla \theta(s)\|_{2} ds\nonumber
\\ &&+C\int^t_0 ((t\!-\!s)^{-\frac12} +(t-s)^{-\frac1{2\kappa }})\| \u(s)\|_{2}^{2-\frac3{2\kappa }} \| A^{\frac\kappa 2}\u(s)\|_{2}^{\frac3{2\kappa }}ds.\nonumber
\bee
Integrating the previous equation and using   (\ref{JkIk}),  we have
\be
 \lefteqn{\int^t_0\|U(s)\|_2ds}\\
 &\le &  \int^t_0\|  e^{-s\L}U_0\|_{2}ds\nonumber
+C\int^t_0\int^r_0\!\!\!(r\!-\!s)^{-\frac12}\|\u(s)\|_2^{1-\frac1{2\kappa }}\|A^{\frac\kappa 2}\u(s)\|_{2}^{\frac1{2\kappa }} \| \nabla \theta(s)\|_{2} dsdr
\\ &&+C\int^t_0\int^r_0 ((r\!-\!s)^{-\frac12} +(r-s)^{-\frac1{2\kappa }} )\| \u(s)\|_{2}^{2-\frac3{2\kappa }} \| A^{\frac\kappa 2}\u(s)\|_{2}^{\frac3{2\kappa }}dsdr.\nonumber
\ee
Changing the integration order and integrating with respect to $r$, we have
\bbe\nonumber
 \lefteqn{\int^t_0\|U(s)\|_2ds}\\&\le &  \int^t_0\|  e^{-s\L}U_0\|_{2}ds\nonumber+C\int^t_0\int^t_s\!\!\!(r\!-\!s)^{-\frac12}\|\u(s)\|_2^{1-\frac1{2\kappa }}\|A^{\frac\kappa 2}\u(s)\|_{2}^{\frac1{2\kappa }} \| \nabla \theta(s)\|_{2} drds
\\ \nonumber &&+C\int^t_0\int^t_s ((r\!-\!s)^{-\frac12}+(r-s)^{-\frac1{2\kappa }} ) \| \u(s)\|_{2}^{2-\frac3{2\kappa }} \| A^{\frac\kappa 2}\u(s)\|_{2}^{\frac3{2\kappa }}drds\nonumber
\\ \nonumber
&\le &  \int^t_0\|  e^{-s\L}U_0\|_{2}ds\nonumber
+Ct^\frac12\int^t_0\|\u(s)\|_2^{1-\frac1{2\kappa }}\|A^{\frac\kappa 2}\u(s)\|_{2}^{\frac1{2\kappa }} \| \nabla \theta(s)\|_{2} ds
\\ &&+C(t^\frac12+t^{1-\frac1{2\kappa }})\int^t_0 \| \u(s)\|_{2}^{2-\frac3{2\kappa }} \| A^{\frac\kappa 2}\u(s)\|_{2}^{\frac3{2\kappa }}\label{decay2}
ds.
\bee
Applying H\"older inequality and energy inequality (\ref{energy}), we have
\bbe
\lefteqn{\int^t_0\|U(s)\|_2ds- \int^t_0\|  e^{-s\L}U_0\|_{2}ds}\nonumber
\\
 &\le &Ct^{1-\frac1{4\kappa }}\left(\int^t_0\|\u(s)\|_2^{4\kappa (1-\frac1{2\kappa })}\|A^{\frac\kappa 2}\u(s)\|_2^{2 } ds\right)^{\frac1{4\kappa }}\left(\int^t_0\|\nabla \theta(s)\|_2^{2 } ds\right)^{\frac1{2}}\hspace{5mm}\nonumber %\label{PPp3}
\\
 & &+C( t^{2-\frac5{4\kappa }}+t^{\frac32-\frac3{4\kappa }})\left(\int^t_0\| \u(s)\|_{2}^{\frac{8\kappa }3-2}\| A^{\frac\kappa 2}\u(s)\|_2^{2} ds\right)^{\frac3{4\kappa }}\nonumber
 \\
  & \le &C( t^{1-\frac1{4\kappa }}+t^{2-\frac5{4\kappa }}+t^{\frac32-\frac3{4\kappa }})\|U_0\|_2^2.\label{decay}
\bee
It follows from (\ref{energy}) that
\bbe \|U(t)\|_2 \le \frac1t\int^t_0 \|U(s)\|_2ds \label{decay5}\bee
Therefore, equation  (\ref{decay}) becomes
\bbe
 \|U(t)\|_2 & \le &\frac1t\int^t_0\|  e^{-s\L}U_0\|_{2}ds+( t^{-\frac1{4\kappa }}+t^{-(\frac5{4\kappa }-1)}+t^{-(\frac3{4\kappa }-\frac12)})\|U_0\|_2^2,\label{decay1}
\bee
which gives the desired  $L^2$ decay (\ref{ass1}) due to (\ref{X2}).

\section{Proof of $L^2$ algebraic decay property (\ref{ass2})}

Observing
\be
\frac5{4\kappa }-1\ge \frac1{4\kappa }, \,\,\,  \frac3{4\kappa }-\frac12\ge \frac1{4\kappa }, \,\mbox{ since } \kappa \le 1,
\ee
and  the assumption (\ref{con1}) or
\bbe \|e^{-t\L}U_0\|_2 \le C t^{-\beta},\label{decay6}\bee
we obtain from  (\ref{decay1}) that
\bbe \|U(t)\|_2 \le C t^{-\beta} + C t^{-\frac1{4\kappa }}.
\bee
Here, we only consider the case $t\ge 1$,  due to the uniform bound  \be\|U(t)\|_2\le \|U_0\|_2\ee given by  (\ref{energy}). Hence, we obtain (\ref{ass2}) when $\beta \le \frac1{4\kappa }$.

For $\beta >\frac1{4\kappa }$,
since $1\ge \kappa  >\frac34$, we set $p>1$ and $q>1$ so that
\bbe \label{dd}\frac1p+ \frac3{4\kappa }=1 \,\, \mbox{ and } \frac1q+ \frac12+\frac1{4\kappa }=1.
\bee
This allows us to employ (\ref{decay6}) and  H\"older inequality in   (\ref{decay2}) to obtain
\bbe\nonumber
\lefteqn{\|U(t)\|_2}\\
 &\le& Ct^{-\beta} +C( t^{-\frac12}+t^{-\frac1{2\kappa }})\nonumber
 \left(\int^t_0\| \u\|_{2}^{(2-\frac3{2\kappa })p}ds\right)^{\frac1p}\left(\int^t_0\| A^{\frac\kappa 2}\u\|_2^{2}ds \right)^{\frac3{4\kappa }}
 \\
  & & +Ct^{-\frac1{2}}\left(\int^t_0\|\u\|_2^{(1-\frac1{2\kappa })q}\right)^{\frac1q}\left(\int^t_0\|A^{\frac\kappa 2}\u\|_2^2 ds\right)^{\frac1{4\kappa }}\left(\int^t_0 \| \nabla \theta\|_{2}^2  ds\right)^{\frac12}\nonumber %l{decay7}
   \\ &\le& Ct^{-\beta} + Ct^{-\frac12}
 \left(\int^t_0\| \u\|_{2}^{(2-\frac3{2\kappa })p}ds\right)^{\frac1p}
  +Ct^{-\frac1{2}}\left(\int^t_0\|\u\|_2^{(1-\frac1{2\kappa })q}\right)^{\frac1q},\label{decay7}
 \bee
where we have used  the uniform bound in (\ref{energy}).

 The desired decay estimate will be derived by a  boot strap iteration scheme. To do so, we begin with the step   $\beta_1=\frac1{4\kappa }$ so that
 \bbe \|U(t)\|_2\le C t^{-\beta_1}.\bee
  For an  integer $n\ge 1$, assuming the existence of a number  $0<\beta_n <\beta\le \frac12$    so that
\bbe \|U(t)\|_2 \le C t^{-\beta_n},\label{betan}\bee
we will show  that the previous estimate can be advanced  to the one involving  a  power    $-\beta_{n+1}$ instead of $-\beta_n$. Employing  (\ref{energy}) and (\ref{dd}), we estimate (\ref{decay7}) by using (\ref{betan}) to obtain  that
\bbe\nonumber
\|U(t)\|_2
 &\le& C [t^{-\beta} +t^{-\frac12+\frac1p -\beta_n(2-\frac3{2\kappa })}\nonumber
  +t^{-\frac1{2}+\frac1q-\beta_n(1-\frac1{2\kappa })}]%\label{decay7}
  \\
  &=& C[t^{-\beta} +t^{\frac12-\frac3{4\kappa } -\beta_n(2-\frac3{2\kappa })}\nonumber
  +t^{-\frac1{4\kappa }-\beta_n(1-\frac1{2\kappa })}]
  \\
  &=& C[t^{-\beta} \nonumber
  +t^{-\frac1{4\kappa }-\beta_n(1-\frac1{2\kappa })}(t^{\frac12-\frac1{2\kappa } -\beta_n(1-\frac1{\kappa })}+1)]
  \\
  &=& C[t^{-\beta} \nonumber
  +t^{-\frac1{4\kappa }-\beta_n(1-\frac1{2\kappa })}(t^{-(\frac1{\kappa }-1)(\frac12 -\beta_n)}+1)]
  \\
  &\le & C(t^{-\beta}
  +t^{-\frac1{4\kappa }-\beta_n(1-\frac1{2\kappa })})\label{a77}
 \bee
 due to $\beta_n \le \frac12$ and $t\ge 1$.
Here we have used the inequalities
\bbe \beta_n(2-\frac3{2\kappa })p< \frac{\frac12(2-\frac3{2\kappa })}{1-\frac3{4\kappa}}=1,
\bee
and
\bbe \beta_n(1-\frac1{2\kappa })q< \frac{\frac12(2-\frac3{2\kappa })}{\frac12-\frac1{4\kappa}}\le 1.
\bee
 If $\frac1{4\kappa }+\beta_n(1-\frac1{2\kappa })\ge \beta$, equation  (\ref{a77}) implies the desired estimate (\ref{ass2}). Otherwise,  we set
 \bbe \beta_{n+1} = \frac1{4\kappa }+\beta_n(1-\frac1{2\kappa }),\,\,\, n\ge 1,
\bee
and hence
\be \beta_{n+1}= \frac1{4\kappa }\sum_{m=0}^{n} (1-\frac1{2\kappa })^m.
\ee
Since the upper limit of $\beta_{n+1}$ is
\bbe \frac1{4\kappa }\sum_{m=0}^\infty (1-\frac1{2\kappa })^m= \frac{\frac1{4\kappa }}{ 1-(1-\frac1{2\kappa })}=\frac12,
\bee
there is an integer $n_0\ge 1$ so that $\frac12 >\beta_{n_0+1} \ge \beta>\beta_{n_0}$  and
\bbe \|U(t)\|_2\le C(t^{-\beta} + t^{-\beta_{n_0+1}} )\le C t^{-\beta}.\bee
This gives (\ref{ass2}) and completes the proof of Theorem \ref{Th1}.

\

\noindent {\bf Acknowledgment.}
This research was supported by The Shenzhen Natural Science Fund of China (the Stable Support Plan Program No. 20220805175116001).


\begin{thebibliography}{99}

%\bibitem{B1} B\acute{e}nard H. Les tourbillons cellulaires dans une nappe liquide. Rev G\`{e}n Sci Pure Appl 1900;11:1261-71.
\bibitem{B2}  H. B\'{e}nard, Les tourbillons cellulaires dans une nappe liquide. Rev. G\'{e}n. Sci. Pures Appl. 11 (1900), 1309-1328.

\bibitem{Miya1990} W. Borchers and T. Miyakawa,
Algebraic $L^2$ decay for Navier-Stokes
flows in exterior domains.  Acta Math. 165 (1990), 189-227.

\bibitem{Miya1992} W. Borchers and T. Miyakawa, $L^2$-Decay for Navier-Stokes Flows
in Unbounded Domains,
with Application to Exterior Stationary Flows. Arch. Rational Mech. Anal. 118 (1992), 273-295.

\bibitem{Miya1995} W. Borchers and T. Miyakawa, On stability of exterior
stationary Navier-Stokes flows. Acta Math. 174 (1995), 311-382.


\bibitem{Chen1992}Z. -M. Chen, Y. Kagei and T. Miyakawa, Remarks on stability of purely conductive
steady states to the exterior Boussinesq problem, Adv. Math. Sci. Appl. 1 (1992), 411-430.

\bibitem{Chen1993} Z. -M. Chen, Solutions of the stationary
and nonstationary
Navier-Stokes equations in exterior domains, Pacific J. Math. 159 (1993), 227-240.

%\bibitem{Chen2005} Z. -M. Chen and W. G. Price, Onset of Chaotic Kolmogorov Flows
%Resulting from Interacting Oscillatory Modes. Commun. Math. Phys. 256 (2005), 737-766.

%\bibitem{Chen2021} Z, -M, Chen, Steady-State Bifurcation of a Non-parallel Flow Involving
%Energy Dissipation over a Hartmann Boundary Layer. J. Nonl. Sci. (2021) 31:91
\bibitem{Chen2012} Z. -M. Chen, A vortex based panel method for potential flow simulation
around a hydrofoil, J. Fluid Structure 28 (2012), 378-391
\bibitem{Chen2019} Z. -M. Chen, Straightforward integration for free surface Green function and body
wave motions, European J Mech./B Fluids 74(2019) 10-18

\bibitem{ChenM} Z. -M. Chen and T. Miyakawa, Decay properties of weak solutions to a perturbed Navier-Stokes system in $R^n$, Adv. Math. Sci. Appl. 7 (1997), 741-770.

\bibitem{Chen2006} Z. -M. Chen and  W. G. Price, On the relation between Rayleigh-B\'{e}nard
convection and Lorenz system.
 Chaos, Solitons and Fractals 28 (2006), 571-578.

\bibitem{F} R. Finn, On the exterior stationary problem for the Navier-Stokes equations,
and associated perturbation problems. Arch. Rational Mech. Anal. 19 (1965),
363-406.

\bibitem{Han} P. Han, Decay rates for the incompressible Navier-Stokes flows in 3D exterior domains. J. Diff. Equs.
263 (2012), 3235-3269.

\bibitem{Henry} D. Henry, Geometric Theory of Semilinear Parabolic Equations, Springer, New York, 1981.

\bibitem{He} C. He and T. Miyakawa, On weighted-norm estimates for nonstationary incompressible Navier-Stokes flows in a 3D exterior domain. J. Diff. Equs.
 246 (2009), 2355-2386.

\bibitem{H94} T. Hishida, Asymptotic behavior and stability of solutions to the exterior convection
problem. Nonlinear Anal. 22 (1994), 895-925.

\bibitem{HY} T. Hishida and Y. Yamada, Global solutions for the heat convection equations in an exterior domain. Tokyo J.
Math. 15 (1992), 135-151.

\bibitem{H97} T. Hishida, On a Class of Stable Steady Flows to the Exterior
Convection Problem. J. Diff. Equs. 141 (1997), 54-85.

\bibitem{H} T. Hishida, Stability of time-dependent motions for fluid-rigid ball
interaction. 	J. Math. Fluid Mech. 26 (2024), article number 17.

\bibitem{Katobook} T. Kato, Perturbation Theory for Linear Operators, Springer, New York, 1995.

\bibitem{K1} H. Kozono, Asymptotic stability of large solutions with large
perturbation to the Navier-Stokes equations. J.  Funct. Anal. 176 (2000), 153-197.

\bibitem{K} H. Kozono and T.  Ogawa,  On stability of Navier-Stokes flows in exterior domains. Arch. Rat. Mech.
Anal. 128 (1994), 1-31.
%\bibitem{Kr} S. G. Krein, Linera Diffrential Equations in Banach Spaces, American Mathematical Society, Providence, 1971

\bibitem{Lady} O. Ladyzhenskaya, The Boundary Value Problems of Mathematical Physics, Springer, New York, 1985.

\bibitem{L}E. N. Lorenz, Deterministic nonperiodic flow. J. Atmos. Sci. 20 (1963), 130-141.

\bibitem{Miya82} T. Miyakawa,
On nonstationary solutions of the Navier-Stokes
equations in an exterior domain.  Hiroshima Math. J. 12 (1982), 115-140.

\bibitem{Miya1988} T. Miyakawa and H.  Sohr,  On energy inequality, smoothness and large time
behavior in $L^2$ for weak solutions of the Navier-Stokes equations in exterior domains.
Math. Z. 199 (1988), 455-478.

%\bibitem{Hardy} G. Hardy and J. Littlewood, Some properties of fractional integral I. Math. Z. (1927), 565--606.
%\bibitem{Sobolev} S. Sobolev, On a theorem in functional analysis (in Russian). Mat. Sob. 46 (1938), 471--497.
\bibitem{1975Masuda}
K. Masuda, On the stability of incompressible viscous
fluid motions past objects. J. Math. Soc. Japan, Vol. 27, No. 2, 1975.

%\bibitem{interpolation} S Krein,  Linear Differential Equations in Banach Spaces. Amer. Math. Soc., Providence, 1972.

\bibitem{O} C. W. Oseen,  Wber die Stokessche Formel und tiber eine Verwandte Anfgabe in der Hydrodynamik, Arkiv. Mat. Astron. Fys. vol. 6, 1910.

\bibitem{R} P. H. Rabinowitz, Existence and nonuniqueness of rectangular solutions of the B\'enard
problem. Arch. Rational Mech. Anal. 29 (1968), 32-57.

\bibitem{B3} L. Rayleigh,  On convection currents in a horizontal layer of fluid, when the higher temperature is on the under side. Philos. Mag.
32 (1916), 529-546.

\bibitem{Simon}  L. de Simon,
 Un{'}applicazione della teoria degli integrali
 singolari allo studio delle equazioni differenziali
 lineari astratte del primo ordine.
 Rend.  Sem. Mat. Univ. Padova  34 (1964), 205-223.

\bibitem{T}R. Teman, Navier-Stokes Equations, North-Holland, Amsterdam, 1977.

\bibitem{Tr}   H.  Triebel,   Interpolation  Theory,  Function  Spaces,  Differential  Operators,
North-Holland, Amsterdam  (1978).

\bibitem{Yosida} K. Yosida, Functional Analysis, Sixth Edition, Springer, New York, 1980.

\end{thebibliography}
\end{document}